\documentclass[3p,12pt]{amsart}
\usepackage{etex}
\usepackage[margin=1.25in]{geometry}
\usepackage[latin1]{inputenc}
\usepackage{amsmath}
\usepackage{amsthm}
\usepackage[alphabetic]{amsrefs}
\usepackage{amsfonts}
\usepackage[all,cmtip]{xy}
\xyoption{rotate}
\usepackage{amsfonts}
\usepackage{amssymb}
\usepackage{amscd}
\usepackage{bbm}
\usepackage{pictexwd,dcpic}
\usepackage[T1]{fontenc}
\usepackage{lmodern}
\usepackage{setspace}
\usepackage{mathpartir}
\usepackage{enumerate,multicol}
\usepackage[svgnames]{xcolor}
\usepackage[colorlinks=true, citecolor=Dark Green, linkcolor=Maroon, urlcolor=Purple]{hyperref}
\usepackage{stmaryrd}
\def\FV{\text{FV}}
\def\obL{\text{ob}\L}
\def\obLe{\emph{\text{ob}}\L}

\def\Lrg{\mathcal{L}_{\text{rg}}}
\def\dep{\text{dep}}
\def\T{\mathbb{T}}

\def\C{\mathcal{C}}
\def\D{\mathcal{D}}
\def\L{\mathcal{L}}
\def\M{\mathcal{M}}
\def\StrucL{\textbf{Str}_\L}
\def\SStrucL{\textbf{SatStr}_\L}

\def\Lcat{\mathcal{L}_{\text{cat}}}

\def\Tcat{\mathbb{T}_{\text{cat}}}

\def\eqdef{=_{\text{df}}}
\def\Ind{\text{Ind}}
\def\Inde{\emph{\text{Ind}}}
\def\FormL{\textbf{Form}_\L}
\def\idtoiso{\mathtt{idtoiso}}
\def\SatGK{\textbf{Sat}(\MK)}

\def\conteq{\sim_c}
\def\yoneda{\textbf{y}}
\def\Iso{\text{Iso}}
\def\Yso{\text{Yso}}
\def\MK{K^\M}
\def\heightof{\textrm{h}}

\def\HomL{\text{Hom}_{\L}}
\def\isrfibsurjK{\mathtt{sps}}
\def\isrfibsurj{\mathtt{fibsurj}}

\def\N{\mathcal{N}}

\def\FibSurj{\textbf{FibSurj}}

\def\UniCat{\textbf{UniCat}}
\def\refl{\text{refl}}
\def\tran{\text{tran}}
\def\sym{\text{sym}}
\def\eqr{\text{eqr}}
\def\alphaequiv{\equiv_{\alpha}}
\def\MR{R^\M}
\def\isiso{\mathtt{isiso}}
\def\isbijection{\mathtt{bij}}
\def\syneq{=}

\newcommand\frakfamily{\usefont{U}{yfrak}{m}{n}}
\DeclareTextFontCommand{\textfrak}{\frakfamily}

\newcommand{\Pitype}[1]{\underset{#1}{\Pi}}
\newcommand{\Sigmatype}[1]{\underset{#1}{\Sigma}}

\author{Dimitris Tsementzis}

\theoremstyle{plain}\newtheorem{lemma}{Lemma}[section]
\theoremstyle{plain}\newtheorem{cor}[lemma]{Corollary}
\theoremstyle{plain}\newtheorem{theo}[lemma]{Theorem}
\theoremstyle{plain}\newtheorem{pretheorem}[lemma]{Pre-Theorem}
\theoremstyle{plain}
\theoremstyle{definition}\newtheorem{defin}[lemma]{Definition}
\theoremstyle{definition} \newtheorem{exam}[lemma]{Example}
\theoremstyle{remark} \newtheorem{remark}[lemma]{Remark}
\theoremstyle{plain}\newtheorem{prop}[lemma]{Proposition}
\theoremstyle{definition}
\theoremstyle{definition}\newtheorem*{notation}{Notation}
\theoremstyle{definition}
\theoremstyle{definition}

\title{A Higher Structure Identity Principle}
\begin{document}

\date{\today}

\address{Department of Philosophy, Princeton University}
\curraddr{Princeton, NJ 08544, USA}
\email{dtsement@princeton.edu}

\address{Department of Statistics and Biostatistics, Rutgers University}
\curraddr{New Brunswick, NJ 08854, USA}
\email{dt506@rci.rutgers.edu}




\def\eqgap{.2ex}
\def\overgap{.4ex}
\def\inferrulerule{.2pt}

\newlength\rulelength
\newlength\toplength
\newlength\bottomlength

\newcommand\myinferrule[2]{%
  \stackMath%
  \setlength\bottomlength{\widthof{$#1$}}%
  \setlength\toplength{\widthof{$#2$}}%
  \ifdim\toplength>\bottomlength%
    \setlength\rulelength{\the\toplength}%
  \else%
    \setlength\rulelength{\the\bottomlength}%
  \fi%
  \mathrel{%
    \stackunder[\overgap]{%
      \stackon[\overgap]{%
        \stackanchor[\eqgap]%
          {\rule{\the\rulelength}{\inferrulerule}}%
        {\rule{\the\rulelength}{\inferrulerule}}%
      }{#2}%
    }{#1}%
  }%
}

\newcommand{\SMod}[2]{\textbf{SatMod}_{#1}^{#2}}
\newcommand{\SMode}[2]{\textbf{\emph{SatMod}}_{#1}^{#2}}
\newcommand{\SStruc}[2]{\textbf{SatStruc}_{#1}^{#2}}
\newcommand{\Struc}[1]{\textbf{Struc}_{#1}}
\newcommand{\Mod}[1]{\textbf{Mod}_{#1}}
\newcommand{\truncate}[2]{\vert\vert #1 \vert\vert_{\mathbf{#2}}}

\begin{abstract}
We prove a Structure Identity Principle for theories defined on types of $h$-level $3$
by defining a general notion of saturation for a large class of structures definable in the Univalent Foundations.
\end{abstract}

\maketitle

Formalizing mathematics in the framework of the Univalent Foundations \cite{HTT} presents unique challenges and opportunities.
One of the main opportunities is the ability to formalize higher-level mathematics (categories, higher categories etc.) in an invariant way by imposing appropriate saturation conditions. This is done, for example, in \cite{HTT} for category theory (cf. Definition 9.1.6).
An interesting challenge is how to express general saturation conditions that apply to wide classes of structures that can be formalized in UF.
The main contribution of this paper is to provide such a general definition of a saturation condition for a wide class of definable structures and to use it to prove a Structure Identity Principle
for all ``category-level'' (or ``$3$-level'') such structures.

To illustrate, consider the following theorem from \cite{HTT}:
\begin{theo}[\cite{HTT}, Theorem 9.4.16]
For any univalent categories $\C$ and $\D$, the type of categorical equivalences $\C \simeq_{\text{\emph{cat}}} \D$ is equivalent to $\C =_{\emph{\UniCat}} \D$.
\end{theo}
\noindent By regarding univalent categories as a ``saturated'' version of an unsaturated structure (i.e. of precategories) we can regard the above result as a specific instance of a more general result of the following form:
\begin{pretheorem}\label{satpretheo}
For any saturated models $\M$ and $\N$ of an $\L$-theory $\T$, the type of $\L$-equivalences $\M \simeq_\L \N$ is equivalent to $\M =_{\SMode{\T}{}} \N$. 
\end{pretheorem}
The purpose of this paper is to make precise and prove a result of this form, such that Theorem 9.4.16 follows as a special case (as an assurance of adequacy).

This is done as follows.
 An ``$\L$-theory $\T$'' will be given by a theory over a FOLDS signature $\L$ in the sense of Makkai \cite{MFOLDS}, i.e. a finite inverse category $\L$. 
A ``model $\M$'' of $\T$ (or just simply an ``$\L$-structure'') will be given by an interpretation of FOLDS into HoTT as has been described in \cite{TsemHMT}.
An ``$\L$-equivalence'' between two such models $\M$ and $\N$ will be given by Makkai's notion of FOLDS $\L$-equivalence, which in set theory is defined as the existence of a span of fiberwise surjective $\L$-homomorphisms from $\M$ to $\N$.
What remains is to define what a ``saturated'' model is in this setting and indeed this is the main contribution of this paper.

The general definition of saturation is based on the following intuitive idea: a structure is saturated if any two indistinguishable components in it are identical.
To make this idea precise we will define for any FOLDS signature $\L$ a general notion of isomorphism for any sort $K$ in $\L$ and any $x,y \colon K$ in the form of an $\L$-formula $x \cong y$. 
The intuition behind $x \cong y$ is that ``$x$ and $y$ are indistinguishable by any sort which depends on $K$, in any position''. Spelling out exactly how ``indistinguishable'' and ``position'' are to be formalized constitutes the main task of the syntactic portion of the paper.
Roughly, this is done by a mutually inductive definition of a formula $\Ind(x,y)$ expressing indistinguishability (up to equivalence) with respect to sorts ``upstairs'' (Definition \ref{ind}) and then a notion of equivalence which says that the indistinguishability relation $\Ind (x,y)$ is a bijection (up to indistinguishability) (Definition \ref{equivalence}). 

We can then interpret the formula $x \cong y$ for any $\L$-structure $\M$ and given this we can define what it means for such an $\M$ to be saturated at any sort $K$ and also associated notions of total saturation, saturation at a given level etc. This forms the key definition of this paper (Definition \ref{saturation}) which provides a very wide notion of saturation of which the known examples of saturated structures become special instances. 
For example, we will show that the totally saturated models of the theory of categories are the univalent categories (Proposition \ref{unicatsat}) whereas precategories are the ``$2$-saturated'' such models (Proposition \ref{precatsat2}).

Finally, given this general notion of saturation, we will prove a precise version of Pre-Theorem \ref{satpretheo} for FOLDS signatures $\L$ of ``height'' 3. 
Signatures of this height can be thought of as expressing theories whose saturated models are given by structures on $1$-types, i.e. on types of $h$-level 3. 
The theory of precategories is an example of such a theory.
This is done in Theorem \ref{HSIP}, the main result of this paper, which can be thought of as a Structure Identity Principle for structures of $h$-level $3$.


\subsection*{Outline of the Paper} In Section \ref{prelim} we fix notation and introduce the relevant concepts from FOLDS.
In Section \ref{syndef} we give the definition of $x \cong y$ for any FOLDS signature $\L$ and any $x ,y \colon K$ (Definitions \ref{ind} and \ref{equivalence}).
In Section \ref{secsat} we define saturation for FOLDS $\L$-structures in type theory (Definition \ref{saturation})
 and prove certain useful properties of our definitions. 
Finally, in Section \ref{HSIPsec} we internalize the notion of FOLDS equivalence in type theory and use this definition to prove a Structure Identity Principle for $3$-level structures (Theorem \ref{HSIP}). We also provide proofs that our higher Structure Identity Principle gives the right answer in the case of the theory of categories (Propositions \ref{precatsat2} and \ref{unicatsat}).

\section{Preliminaries and Notation}\label{prelim}


\subsection{FOLDS}

By FOLDS we understand the system of Makkai introduced in \cite{MFOLDS}. We here introduce the relevant concepts.
A \emph{FOLDS signature} $\L$ is a finite inverse category. 
We assume each signature $\L$ comes with a \emph{level} function $l \colon \obL \rightarrow \mathbb{N}$ defined as follows
\[
l(K) = \left\{
	\begin{array}{ll}
		1  & \mbox{if $K$ is the codomain only of $1_K$} \\
		\underset{f \colon K_f \rightarrow K}{\text{sup}}l(K_f)+1 & \mbox{otherwise} 
	\end{array}
\right.
\]
Note that ``top-level'' sorts are assigned level $1$, contrary to Makkai's original definition. This is done to foreshadow that these sorts, when saturated, will be interpreted as types of $h$-level 1. However, we will write $K > K'$ to denote that $K$ has a \emph{lower} level than $K'$ (i.e. that $l(K)<l(K')$ in order to think of $K>K'$ as saying that $K$ is ``higher'' (and therefore depends on) $K'$. 
We also write $\heightof(\L)$ for the \emph{height} of $\L$, i.e. the maximum level of any sort $K$ in $\L$ (since we are assuming that $\L$ is finite, the height will also always be finite).
For any $K \in \obL$ we will write $\L \downarrow K$ for the comma category under $K$ in $\L$.
We will generally use the notation $\L(K,K')$ for the hom-sets of $\L$.

Variables for $\L$ are given by a functor $V \colon \L \rightarrow \textbf{Set}$ satisfying certain obvious conditions, e.g. that for any $K \neq K'$, $V(K)$ and $V(K')$ are countably infinite disjoint sets etc. 
We will write $x \colon K$ for $x \in V(K)$, for any $K \in \obL$.
For any $x \colon K$ and $f \in \L(K, K_f)$ we write $x_f$ for $V(f)(x)$. It is helpful to think of $f$ as a \emph{position} and $x_f$ as the variable that $x$ depends on in position $f$.
We write $\dep (x)$ for the set $\lbrace x_f \: \vert \: \text{dom}(f)=K \rbrace$, i.e. the set of \emph{dependent variables} of $x$ and $\partial x$ for the the set of dependent variables of $x$ but omitting $x$, i.e the \emph{boundary} of $x$. 

A \emph{context} is a finite subfunctor of $V$.
A \emph{context morphism} is a natural transformation $s \colon \Gamma \rightarrow \Delta$.
For any two contexts $\Gamma$ and $\Delta$ we write $\Gamma \cup \Delta$ for their union as subobjects of $V$, i.e. the functor that takes $K \mapsto \Gamma(K) \cup \Delta(K)$.
It is not hard to see that both $\dep(x)$ and $\partial x$ define contexts.
We will use the symbol $\syneq$ for equality of contexts and variables.

We define the syntax associated to $\L$ by an inductive definition in the usual way (except the clauses for quantifiers are a little more involved that first-order logic) and we write $\FormL (\Gamma)$ for the set of $\L$-formulas in context $\Gamma$.
We write FV($\phi$) for the free variables of an $\L$-formula $\phi$.
Substitution along a context morphism is defined in the expected way and we write $s(\phi)$ for the result of substituting the variables in a formula $\phi$ in context $\Gamma$ along $s \colon \Gamma \rightarrow \Delta$.
As usual, we consider formulas up to \emph{$\alpha$-equivalence}, denoted by $\alphaequiv$, i.e. up to renaming of their bound variables.
A \emph{theory} $\T$ in $\L$ is a set of $\L$-sentences.
It is also useful (for the sake of Definition \ref{ind}) to define the coarser notion of
\emph{contextual equivalence} of formulas $\phi$ (in context $\Gamma$) and $\psi$ (in context $\Delta$), denoted $\conteq$, as follows: $\phi \conteq \psi$ iff there exists an isomorphism $s \colon \Gamma \rightarrow \Delta$ such that $s(\phi) \alphaequiv \psi$. In other words, two formulas are contextually equivalent if we can rename the free variables of one in such a way as to make it $\alpha$-equivalent to the other. 
This notion is helpful in distinguishing repeat uses of the same variable in formulas, e.g. $\phi(x,x)$ is not contextually equivalent to  $\phi(x,y)$ but $\phi(x,y)$ is contextually equivalent to $\phi(z,w)$.

For $R \in \obL$ with $l(R) = 1$ and $\alpha \colon R$ we write $R(\alpha)$ as syntactic sugar for the formula $\exists \alpha \colon R. \top$. This means that $\text{FV}(R(\alpha)) = \partial \alpha$ which means that if $\partial \alpha = \partial \beta$ then $R(\alpha) \alphaequiv R(\beta)$.

For a set of variables $\Gamma$ (not necessarily a context) we write $\forall \Gamma \phi$ for the universal closure of $\phi$ under all variables in $\Gamma$. 
When this is done we will assume that $\forall \Gamma \phi$ is well-formed, i.e. that $\FV(\forall \Gamma \phi)$ is a well-formed context.
For a given set of formulas $S$, by $\bigwedge S$ we will mean the conjunction of all formulas in $S$ (and similarly for $S_1 \wedge S_2$).
We will also reserve the symbol $\Leftrightarrow$ for logical equivalence over a standard deductive system for FOLDS (cf. \cite{MFOLDS} for details). Alternatively, for the purposes of this paper one may take the logical consequence relation $\Rightarrow$ to mean simply derivability in the type-theoretic semantics of FOLDS formulas (cf. \cite{TsemHMT} for a detailed explanation of these semantics). For example without any loss of content one may simply take something like $\phi \Rightarrow \psi$ to mean that given a term of the interpretation of the formula $\phi$ as a type one can produce a term of the type $\psi$ in type theory.  
We will write $\T \models \phi \Leftrightarrow \psi$ to denote $\T$-provable logical equivalence of $\phi$ and $\psi$.

Several examples of FOLDS signatures will come in handy for the purpose of illustration, so we define them here, with the level of each sort displayed on the left.
We will denote by $\Lrg$ the FOLDS signature

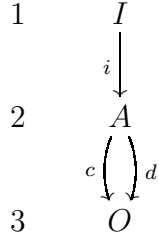
\hspace{5cm} \xymatrix{
1 &I \ar[d]_{i} \\
2 &A \ar@/^/[d]^{d} \ar@/_/[d]_{c} \\
3 &O
}

\noindent subject to the relation $di=ci$. $\Lrg$ can be thought of as a signature useful for formalizing reflexive graphs.

By $\Lrg^{=_A}$ we will denote the FOLDS signature extending $\Lrg$ where we add an ``equality'' on $A$

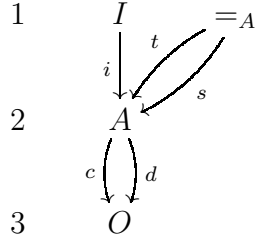
\hspace{5cm} \xymatrix{
1 &I \ar[d]_{i} & =_A \ar@/^/[ld]^{s} \ar@/_/[ld]_{t} \\
2 &A \ar@/^/[d]^{d} \ar@/_/[d]_{c} \\
3 &O
}

\noindent subject to the relations $di=ci$, $ds=dt$, $cs=ct$.

By $\Lcat$ we will denote the following FOLDS signature

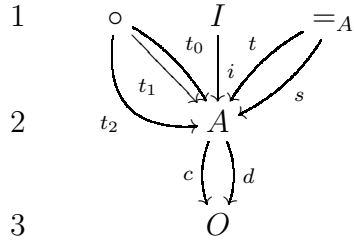
\hspace{5cm} \xymatrix{
1& \circ \ar@/^5pt/[rd]^{t_0} \ar[rd]_{t_1} \ar@/_20pt/[rd]_{t_2} & I \ar[d]^{i}  &=_A \ar@/_/[ld]_{t} \ar@/^/[ld]^{s} \\
2 & & A \ar@/_/[d]_{c} \ar@/^/[d]^{d} \\
3 &  & O \\
}

\noindent subject to the same relations as $\Lrg^{=_A}$ in addition to the following:
\[
dt_0=dt_2, ct_1=ct_2, dt_1=ct_0
\]
\[
ds=dt, cs=ct
\]
By $\Tcat$ we will denote the $\Lcat$-theory of categories axiomatized in the usual way, i.e. by adding axioms expressing that $\circ$ is an associative operation for which $I$ picks out a right and left unit etc. (cf. \cite{TsemHMT} for details). 

\subsection{Homotopy Type Theory}

On the semantic side, we will assume that we will work with some Homotopy Type Theory to which we will refer generically as \emph{type theory}.
The system in \cite{HTT} will do just fine, for example.
In particular, we will assume that all the types in terms of which $\L$-structures are defined live in some univalent universe $\mathcal{U}$.
The terms \emph{proposition}, \emph{set}, $n$\emph{-type}, $h$\emph{-level}, \emph{precategory}, \emph{univalent category} etc. refer to the standard definitions as given e.g. in \cite{HTT}.
We will denote by \textbf{PreCat} the type of precategories and by \textbf{UniCat} the type of univalent categories.
For a given type $A$ and $a,b \colon A$ we will write $a=_Ab$ for the identity type of $a$ and $b$, often abbreviating to $a=b$ when $A$ is clear from the context.
For a given $f \colon A \rightarrow B$ we will also use the abbreviation $\isbijection(f)$ for the type
\[
\bigg( \Pitype{x,x' \colon A} f(x)= f(x') \rightarrow x=x'\bigg) \times \bigg( \Pitype{y \colon B} \Sigmatype{x \colon A} f(x)=y \bigg)
\]
We use the notation $\truncate{A}{n}$ for the $n$-truncation of a type $A$, where $n$ will denote the $h$-level, e.g. for $n=1$, $\truncate{A}{1}$ is the propositional truncation.

Any FOLDS signature $\L$ can be translated as data in type theory in a straightforward way, giving us a notion of a type $\StrucL$ of $\L$-structures a term $\M$ of which can be thought of as an $\L$-structure (cf. \cite{TsemHMT} for details).
This is done by extracting a $\Sigma$-type from $\L$ by induction on the level of $\L$.
For example, we have
\[
\Struc{\Lrg} \eqdef \Sigmatype{O \colon \mathcal{U}} \: \Sigmatype{A \colon O \rightarrow O \rightarrow \mathcal{U}} \: \Pitype{x \colon O} A(x,x) \rightarrow \mathcal{U}
\]
Whenever we pick a specific $\L$-structure $\M \colon \Struc{\L}$ we will denote by $K^\M$ the interpretation of any sort $K \in \obL$. For example, the data of $\M \colon \Struc{\Lrg}$ can be written as 
$
\langle O^\M, A^\M, I^\M \rangle
$.
To further illustrate the notation, we can for instance define
\[
\Struc{\Lrg^{=_A}} \eqdef \Sigmatype{\M \colon \Struc{\Lrg}} \:\: \bigg( \Pitype{x,y \colon O^\M} A^\M (x,y) \times A^\M (x,y) \rightarrow \mathcal{U} \bigg)
\]
Similarly, all $\L$-formulas can be translated under the usual propositions-as-types translation of the quantifiers and logical connectives (with $\exists$ interpreted as truncated $\Sigma$ and $\vee$ interpreted as truncated $+$ ensuring that the interpretations are all $h$-props).
For example, for the $\Lrg$-sentence $\exists x \colon O. \forall f \colon A(x,x). I(f)$ will be interpreted in an $\L$-structure $\M$ as the proposition
\[
\truncate{\Sigmatype{x \colon O^\M} \Pitype{f \colon A^\M(x,x)} I^\M (f)}{1}
\]
When necessary we will write $\phi^\M$ to refer to the interpretation of a certain $\L$-formula into an $\L$-structure $\M$. 
We can then obtain a notion of \emph{satisfaction} of a formula $\phi$ in an $\L$-structure $\M$ and therefore a notion of a \emph{model} $\M$ of an $\L$-theory $\T$ (where a \emph{theory} $\T$ is a set of $\L$-sentences).
We write $\Mod{\T}$ for the type of $\T$-models, i.e. those $\L$-structures that satisfy all $\phi \in \T$.
For example, for the $\Lrg$-theory $\T_1 = \lbrace \exists x \colon O. \forall f \colon A(x,x). I(f) \rbrace$ we have
\[
\Mod{\T_1} \eqdef \Sigmatype{\M \colon \Struc{\L}} \truncate{\Sigmatype{x \colon O^\M} \Pitype{f \colon A^\M(x,x)} I^\M (f)}{1}
\]
Finally, note that a well-formed context $\Gamma$ in the sense of FOLDS clearly gives rise to a well-formed context $\Gamma^\M$ in type theory, and so we will freely use the notation $\Gamma^\M$ to denote the interpretation of $\Gamma$ in some $\L$-structure $\M$.

\subsection{Note on Metatheory}

Throughout this paper we work in a set-theoretic meta-theory.
This should already be clear by our use of e.g. the symbol $\in$ in our definition of a FOLDS signature and its associated notions.
However, we aim for this set-theoretic metatheory to be ``foundation-agnostic'' in that our definitions and arguments could equally well be formalized within an extensional set theory like ZF or a suitable model of a structural set theory inside the Univalent Foundations.

One further clarification is in order, to avoid any confusion.
Since our metalanguage is set-theoretic we will often refer to expressions in type theory using set-theoretic abbreviations. 
For example, the symbol $$\underset{s \in S}{\times} R_s$$ might refer to a well-formed type in type theory (a non-dependent sum with $S$-many components) but is not itself to be understood as a well-formed type in type theory. Rather, it is a description of such a type in our set-theoretic metalanguage (hence the use of ``$\in$'' under $\times$).
Similarly, in Section \ref{syndef} below we will construct set-theoretically certain FOLDS formulas which we then use to describe certain well-formed types in type theory. But our FOLDS formulas, as constructed, are not themselves to be understood as constructed \emph{inside} type theory. In particular, our description of the relevant formulas will contain certain distinctions between variables that will be internally invisible to the type theory.

\section{A Syntactic Definition of Isomorphism}\label{syndef}

Fix a FOLDS signature $\L$.
Let $x \colon K$ for some $K \in \obL$. Then we define the set of \emph{$x$-compatible} sorts as follows:
\[
\L_x \eqdef \lbrace R \in \obL \:\vert\: R > K \text{ and }\: \forall q \in \L(R ,K) \: \forall p_1, p_2 \in \L(K, K'), p_1q = p_2q \Rightarrow x_{p_1} = x_{p_2} \rbrace
\]
If $R \in \L_x$ we say that $x$ is \emph{$R$-compatible}.
The idea is that a certain variable is not $R$-compatible if its dependent variables violate some identity that $R$ requires in order to be defined. The following example illustrates this.

\begin{exam}
In $\Lrg$ a variable $f \colon A$ with $f_d = x$ and $f_c = y$ (and where $x \neq y$) is not $I$-compatible. This is because $di=ci$ yet $f_c \neq f_d$ as just stated. However, in $\Lrg^{=_A}$ $f$ is $=_A$-compatible (since $=_A$ does not ``require'' an identification between ``source'' and ``target'').
\end{exam}

\begin{remark}
If $\partial x = \partial y$ then $\L_x = \L_y$. In other words, variables with the same boundary are compatible with exactly the same sorts.
\end{remark}

The idea of the syntactic definition of isomorphism for $K$ goes as follows. For any variables $x, y \colon  K$ we will define by mutual induction a formula 
\[
\Ind (x,y) \in \FormL (\dep (x) \cup \dep (y))
\] 
that expresses the fact that $x$ and $y$ are ``indistinguishable up to equivalence'' by any of the sorts they are jointly compatible with (i.e. for all $R \in \L_x \cap \L_y$) and where ``equivalence'' will be defined, for any $\alpha, \beta \colon A$, as a formula 
\[
A(\alpha) \simeq A (\beta) \in \FormL (\partial \alpha \cup \partial \beta)
\]
Our first goal is to define these formulas (by mutual induction).


We begin with the definition of indistinguishability, i.e. of the formula $\Ind(x,y)$. Let $x, y \colon K$ and $R \in \L_x \cap \L_y$ and $p \in \L(R,K)$. Then we have the following two  definitions expressing respectively ``$x$ and $y$ cannot be distinguished by $R$ in position $p$, up to equivalence''
and ``$x$ and $y$ cannot be distinguished by $R$ in any position, up to equivalence'':
\begin{align*}
\Ind_R^p (x,y) &\eqdef \bigwedge \lbrace \forall \Gamma_{\alpha,\beta} R(\alpha) \simeq R(\beta) \: \vert \: \alpha, \beta \colon R, \alpha_p = x, \beta_p = y, \forall p \neq q \in \L(R, K). \alpha_q=\beta_q, \\ 
&\:\:\:\:\:\:\:\:\:\:\:\:\:\:\:\:\Gamma_{\alpha, \beta} = \partial \alpha \cup \partial \beta \setminus (\dep(\alpha_p) \cup \dep(\beta_p)) \rbrace / \conteq \\
\Ind_R (x,y) &\eqdef \bigwedge_{p \in \L(R, K)} \Ind^p_R (x,y) \\
\end{align*}

\begin{remark}
The use of $\conteq$ in $\Ind_R^p$ ensures that $\Ind_R^p (x,y)$ is finite. 
Strictly speaking, this makes $\Ind_R^p (x,y)$ a ``conjunction'' of equivalence classes of formulas up to contextual equivalence. We brush over the difficulty by assuming we in each case choose a canonical representative for the equivalence class.
\end{remark}

\begin{exam}
Take $\Lrg^{=A}$. Then we have
\[
\Ind_{=_A}^s (f,g) = \lbrace \forall h (f=_Ah \simeq g=_Ah), f=_Af \simeq g=_Af, f=_Ag \simeq g=_Ag \rbrace
\]
where we have used the standard abbreviation $f=_Ag$ for the sort $=_A(f,g)$.
Intuitively, $\Ind^s_A (f,g)$ expresses the fact that ``binary relation'' $=_A$ should not distinguish between $f$ and $g$ when they are ``plugged in'' position $s$, i.e. in the ``source'' position, which we take to be the left-hand side.
And ``does not distinguish'' is in turn understood as the existence of an equivalence (to be defined below) between the sorts $f=_Ah$ and $g=_Ah$ that we obtain by plugging $f$ and $g$ into the $s$ position, for any $h$.
 Clearly this is something that we would expect of the binary relation of equality, although of course at this point this is simply uninterpreted syntax.
\end{exam}

\begin{defin}[$\Ind(x,y), x \cong y$]\label{ind}
For $x,y \colon K$ we define the following formula
\[
\Ind (x,y) \eqdef \bigwedge_{R \in \L_x \cap \L_y} \Ind_R (x,y)
\]
which we express as $x$ and $y$ are \textbf{indistinguishable}.
In the special case when $\partial x = \partial y$ we write $x \cong y$ for $\Ind (x,y)$ and say that $x$ and $y$ are \textbf{isomorphic}.
\end{defin}

Observe that we have $\FV(\Ind(x,y)) = \dep (x) \cup \dep (y)$.
$\Ind(x,y)$ captures the intended meaning that ``$x$ and $y$ are indistinguishable up to equivalence by any $x$-compatible and $y$-compatible $R$, in any position''.
Of course, for this definition to be useful we also need to define equivalence as it appears in $\Ind_R^p (x,y)$ (i.e. in the formula $\forall \Gamma_{\alpha,\beta} R(\alpha) \simeq R(\beta)$), which we do below. But even before we do that, we can draw some useful conclusions in some degenerate cases.

\begin{exam}
In $\Lrg$ let $f,g \colon A(x,y)$. Then $\L_f = \L_g = \varnothing$ which means $f \cong g \alphaequiv \top$ (as the empty conjunction). This is a strange but I think correct conclusion. For if two terms cannot be distinguished by any data in the FOLDS signature (i.e. they have no compatible sorts) then the appropriate saturation condition is to regard them as always isomorphic.
\end{exam}

\begin{exam}
A bit stranger is the situation where, say, $f \colon A(x,x)$ and $g \colon A(x,y)$, still in $\Lrg$. 
In this case, we have $\L_f \neq \varnothing$, but $\L_f \cap \L_g = \varnothing$, and so Definition \ref{ind} dictates that we should regard $f$ and $g$ as always indistinguishable.  Formally, this does not seem to be problematic but it is a little harder to make sense of since there is certainly something in the syntax that distinguishes $f$ and $g$, namely that in the case of the former we have $f_d=f_c$ but the analogous equation does not hold of $g$.
\end{exam}

\begin{exam}\label{eqexam}
In $\Lrg^{=_A}$, let $f,g \colon A(x,y)$. Then $\L_f = \L_g = \lbrace =_A \rbrace$ and we have
\begin{align*}
\Ind(f,g) &= \Ind_{=_A}^s (f,g) \wedge \Ind_{=_A}^t (f,g) \\
&= \lbrace \forall h (f=_Ah \simeq g=_Ah), f=_Af \simeq g=_Af, f=_Ag \simeq g=_Ag \rbrace \wedge \\
 &\: \: \:\:\:\:  \lbrace \forall h (h=_Af \simeq h=_Ag), f=_Af \simeq f=_Ag, g=_Af \simeq g=_Ag\rbrace
\end{align*}
\end{exam}

\begin{prop}\label{degiso}
For $K \in \obLe, l(K) = 1, x, y \colon K$ we have $\emph{\Ind} (x,y) \alphaequiv \top$.
\end{prop}
\begin{proof}
We have $\L_x = \L_y = \varnothing$ hence $\Ind (x,y)$ is an empty conjunction and therefore the same thing as $\top$.
\end{proof}

We now turn our attention to the definition of $\simeq$.

\begin{defin}[$K(\alpha) \simeq K(\beta)$]\label{equivalence}
For $K \in \obL$ and $\alpha, \beta \colon K$ we define the following formula:
\begin{align*}
K(\alpha) \simeq K (\beta) &\eqdef \big[ \forall x \colon K(\alpha) \exists y \colon K(\beta). \Ind (x,y) \wedge (\forall y' \colon K(\beta). \Ind (x,y') \rightarrow y \cong y') \big] \wedge \\
& \:\:\:\:\:\:\:\:\: \big[ \forall x,x' \colon K(\alpha) \forall y,y' \colon K(\beta). \Ind (x,y) \wedge \Ind (x',y') \wedge y \cong y' \rightarrow x \cong x' \big] \wedge \\
& \:\:\:\:\:\:\:\:\: \big[ \forall y \colon K (\beta) \exists x \colon K(\alpha). \Ind (x,y) \big]\\
\end{align*}
When the formula holds we say that $K(\alpha)$ is \textbf{equivalent} to $K(\beta)$.
\end{defin}

To aid understanding, here is the intended meaning behind $K(\alpha) \simeq K (\beta)$: the first conjunct expresses that $\Ind (x,y)$ is a functional relation (up to $\cong$), the second line expresses that $\Ind (x,y)$ is injective (up to $\cong$), and the third line express that $\Ind (x,y)$ is surjective. Putting all this together, $K(\alpha) \simeq K (\beta)$ can be understood as saying ``$\Ind (x,y)$ is an equivalence up to $\cong$''. 

Note that we have $\FV(K(\alpha) \simeq K(\beta))=\partial \alpha \cap \partial \beta$. Indeed, since $K(\alpha) \simeq K(\beta)$ depends only on the respective boundaries of $\alpha$ and $\beta$ it is usually helpful to write $K(\partial \alpha) \simeq K(\partial \beta)$ when the variables in the boundary can be explicitly listed. For example, we can write $A(x,y)$ for the sort of ``arrows'' in $\Lrg$ dependent on $x,y \colon O$ instead of the more alienating $A(f)$ (for some $f$ that depends on $x$ and $y$). 

\begin{exam}
In $\Lrg$ let $x,y \colon O$. Then we have:
\begin{align*}
A(x,x) \simeq A(y,y) &=   \\ \big[ \forall f \colon A(x,x) &\exists g \colon A(y,y). \Ind (f,g) \wedge (\forall h \colon A(y,y). \Ind (f,h) \rightarrow g \cong h) \big] \wedge \\
  \big[ \forall f,f' \colon A(x&,x) \forall g,g' \colon A(y,y). \Ind (f,g) \wedge \Ind (f',g') \wedge g \cong g' \rightarrow f \cong f' \big] \wedge \\
\big[ \forall g \colon A(y,y) &\exists f \colon A(x,x). \Ind (f,g) \big] =\\
\\ \big[ \forall f \colon A(x,x) &\exists g \colon A(y,y). (I(f) \simeq I(g)) \wedge (\forall h \colon A(y,y). (I(f) \simeq I(h)) \\ &\:\:\:\:\:\:\:\:\:\:\:\:\:\:\:\:\:\:\:\:\:\:\:\:\:\:\:\:\:\:\:\:\:\:\:\:\:\:\:\:\rightarrow (I(g) \simeq I(h)) \big] \wedge \\
  \big[ \forall f,f' \colon A(x&,x) \forall g,g' \colon A(y,y). (I(f)\simeq I(g)) \wedge (I(f') \simeq I(g')) \wedge (I(g) \simeq  I(g'))\\ &\:\:\:\:\:\:\:\:\:\:\:\:\:\:\:\:\:\:\:\:\:\:\:\:\:\:\:\:\:\:\:\:\:\:\:\:\:\:\:\:\rightarrow (I(f) \simeq I(f')) \big] \wedge \\
\big[ \forall g \colon A(y,y) &\exists f \colon A(x,x). (I(f) \simeq I(g)) \big]\\
\end{align*}
The only sort that could distinguish two variables $f \colon A(x,x)$ and $g \colon A(y,y)$ in $\Lrg$ is $I$.
Therefore, $\Ind(f,g)$ becomes $(I(f)\simeq I(g))$ as prescribed by the definition of $\Ind$. So what we get, intuitively, is that $A(x,x)$ and $A(y,y)$ are equivalent iff there is a one-to-one correspondence  between ``terms'' $f$ of $A(x,x)$ and ``terms'' $g$ of $A(y,y)$ that give equivalent sorts when plugged into $I$, i.e. such that $I(f) \simeq I(g)$. 
Except this ``one-to-one correspondence'' is one-to-one only up to $\cong$ which in the case of $\Lrg$ amounts to the exact same property, namely $g \cong h$ iff $g$ and $h$ produce equivalent sorts when plugged into the ``type family'' $I$. 
We clearly here want to regard $I$ as a predicate (e.g. picking out which ``arrows'' are ``identities'') which would make the equivalence $I(f) \simeq I(g)$ a logical equivalence. This is what the next proposition establishes for sorts, like $I$, of level $1$.
\end{exam}

\begin{prop}\label{degequiv}
For $K \in \obLe$, $l(K)=1$ we have $K(\alpha) \simeq K(\beta) \Leftrightarrow K(\alpha) \leftrightarrow K(\beta)$
\end{prop}
\begin{proof}
Note that $\L_\alpha = \L_\beta=\varnothing$ which means that for $x,x' \colon K(\alpha)$ and $y',y \colon K(\beta)$ we have 
\[
\Ind_R^p (x,y) \Leftrightarrow \Ind_R^p (x',y) \Leftrightarrow \Ind_R^p (x,y') \Leftrightarrow x \cong x' \Leftrightarrow y \cong y' \Leftrightarrow \top
\]
Therefore, by Proposition \ref{degiso}, we can see that $K(\alpha) \simeq K(\beta)$ is logically equivalent to 
\[
\forall x \colon K (\alpha)\exists y \colon K(\beta).\top \wedge \forall y \colon K(\beta) \exists x \colon K(\alpha). \top
\]
which in turn is logically equivalent to $K(\alpha) \leftrightarrow K(\beta)$.
\end{proof}


Using propositions \ref{degiso} and \ref{degequiv} we can successively obtain formulas for $x \cong y$ for $x,y \colon K$ for sorts $K$ or arbitrarily high level.
We will write $\cong_K$ for the isomorphism relation defined on a sort $K$, or $\cong_{K(\alpha)}$ whenever we want to make the specific context explicit.
We now prove the following crucial result, which is the last general result we will carry out ``formally'', i.e. assuming an arbitrary formal system for FOLDS.

\begin{prop}\label{reflprop}
Let $\L$ be a FOLDS signature. For any $K \in \obL$, $\Inde(x,y)$ is an equivalence relation. That is, it satisfies the following three properties:
\begin{equation}
\tag{$\refl_{\Inde}$} \forall x \colon K. \Ind (x,x)
\end{equation}
\begin{equation}
\tag{$\sym_{\Inde}$} \forall x,y \colon K. \Inde (x,y) \Rightarrow \Inde (y,x)
\end{equation}
\begin{equation}
\tag{$\tran_{\Inde}$} \forall x,y,z \colon K. \Inde (x,y) \wedge \Inde (y,z) \Rightarrow \Inde (y,z)
\end{equation}
In particular, for any $K \in \obL$, $x\cong y $ is an equivalence relation.
\end{prop}.
\begin{proof}
We proceed by induction on the level of $K$. 

\emph{Base Step}: For any sort $K$ with $l(K)=1$ the statement follows immediately from Proposition \ref{degiso}. 


\emph{Inductive Step}: For any choices of $\alpha, \beta, \gamma$ define the following formulas:
\begin{equation}
\tag{$\refl_{\simeq}$}  K(\alpha) \simeq K (\alpha)
\end{equation}
\begin{equation}
\tag{$\sym_{\Ind}$}  K(\alpha) \simeq K(\beta) \Rightarrow K(\beta) \simeq K(\alpha)
\end{equation}
\begin{equation}
\tag{$\tran_{\Ind}$}  K(\alpha) \simeq K(\beta) \wedge K(\beta) \simeq K(\gamma) \Rightarrow K(\alpha) \simeq K(\gamma)
\end{equation}
Let us denote the conjunction of the three statements above as $(\eqr^K_{\simeq})$ and similarly let us denote by $(\eqr^K_{\Ind})$ the conjunction of the three sentences in the statement of the Proposition.
Similarly, let us write $(\eqr_{\Ind}^n)$ for the statement that for any sort $K$ with $l(K) \geq n$, $(\eqr^K_{\Ind})$ holds and $(\eqr^n_{\simeq})$ for the analogous statement for $\simeq$.
The proof of the inductive step now proceeds in two steps. First we prove that $(\eqr^n_{\Ind}) \Rightarrow (\eqr^n_{\simeq})$ and then that $(\eqr^n_{\simeq}) \Rightarrow (\eqr^{n+1}_{\Ind})$.
Put together these give $(\eqr^n_{\Ind}) \Rightarrow (\eqr^{n+1}_{\Ind})$ which is exactly what the inductive step requires us to prove.
We take them in turn.

$(\eqr^n_{\Ind}) \Rightarrow (\eqr^n_{\simeq})$: For $(\refl_{\simeq})$ we have:
\begin{align*}
(\refl_{\simeq})  &\eqdef K (\alpha) \simeq K(\alpha)  \\
				&\eqdef  \big[ \forall x \colon K(\alpha) \exists y \colon K(\alpha). \Ind (x,y) \wedge (\forall y' \colon K(\alpha). \Ind (x,y') \rightarrow y \cong y') \big] \wedge \\
& \:\:\:\:\:\:\:\:\: \big[ \forall x,x' \colon K(\alpha) \forall y,y' \colon K(\alpha). \Ind (x,y) \wedge \Ind (x',y') \wedge y \cong y' \rightarrow x \cong x' \big] \wedge \\
& \:\:\:\:\:\:\:\:\: \big[ \forall y \colon K (\alpha) \exists x \colon K(\alpha). \Ind (x,y) \big] \\
				&\alphaequiv \big[ \forall x \colon K(\alpha) \exists y \colon K(\alpha). x \cong y \wedge (\forall y' \colon K(\alpha). x \cong y' \rightarrow y \cong y') \big] \wedge \\
& \:\:\:\:\:\:\:\:\: \big[ \forall x,x' \colon K(\alpha) \forall y,y' \colon K(\alpha). x \cong y \wedge x' \cong y' \wedge y \cong y' \rightarrow x \cong x' \big] \wedge \\
& \:\:\:\:\:\:\:\:\: \big[ \forall y \colon K (\alpha) \exists x \colon K(\alpha). x \cong y \big] \\
\end{align*}
The last form is obtained because when $x,y$ are of the same sort we write $x \cong y$ instead of $\Ind (x,y)$.
The first conjunct $\forall x \colon K(\alpha) \exists y \colon K(\alpha). x \cong y \wedge (\forall y' \colon K(\alpha). x \cong y' \rightarrow y \cong y')$ follows from the reflexivity of $\cong$, which is given by the inductive hypothesis.
The second conjunct $\forall x,x' \colon K(\alpha) \forall y,y' \colon K(\alpha). x \cong y \wedge x' \cong y' \wedge y \cong y' \rightarrow x \cong x'$ follows from the symmetry and transitivity of $\cong$, which is once again given by the inductive hypothesis.
Finally, the third conjunct $\forall y \colon K (\alpha) \exists x \colon K(\alpha). x \cong y$ follows from the reflexivity of $\cong$ which is once again given by the inductive hypothesis.
For $(\sym_{\simeq})$ we have:
\begin{align*}
(\sym_{\simeq})  &\eqdef K (\alpha) \simeq K(\beta) \Rightarrow K (\beta) \simeq K(\alpha)  \\
				&\eqdef  \bigg[ \forall x \colon K(\alpha) \exists y \colon K(\beta). (\Ind (x,y) \wedge (\forall y' \colon K(\beta). \Ind (x,y') \rightarrow y \cong y')) \wedge \\
& \:\:\:\:\:\:\:\:\: \forall x,x' \colon K(\alpha) \forall y,y' \colon K(\beta). \Ind (x,y) \wedge \Ind (x',y') \wedge y \cong y' \rightarrow x \cong x' \wedge \\
& \:\:\:\:\:\:\:\:\: \forall y \colon K (\beta) \exists x \colon K(\alpha). \Ind (x,y) \bigg] \Rightarrow \\
				&\:\:\:\:\:\:\:\:\: \bigg[ \forall x \colon K(\beta) \exists y \colon K(\alpha). \Ind (x,y) \wedge (\forall y' \colon K(\alpha). \Ind (x,y') \rightarrow y \cong y') \wedge \\
& \:\:\:\:\:\:\:\:\: \forall x,x' \colon K(\beta) \forall y,y' \colon K(\alpha). \Ind (x,y) \wedge \Ind (x',y') \wedge y \cong y' \rightarrow x \cong x' \wedge \\
& \:\:\:\:\:\:\:\:\: \forall y \colon K (\alpha) \exists x \colon K(\beta). \Ind (x,y) \bigg] \\
\end{align*}
By the inductive hypothesis, we can assume that $\Ind$ is symmetric, which means that if $\Ind(x,y)$ defines an ``equivalence up to $\cong$'' then $\Ind(y,x)$ also defines such an equivalence. We can use this basic idea to deduce each of the conjuncts of the RHS from the LHS. We deduce the conjunct $\forall y \colon K (\alpha) \exists x \colon K(\beta). \Ind (x,y)$ on the RHS as an illustration. We have:
\begin{align*}
K (\alpha) \simeq K(\beta) &\Rightarrow \forall y \colon K (\beta) \exists x \colon K(\alpha). \Ind (x,y) \\
					&\Rightarrow \forall y \colon K (\beta) \exists x \colon K(\alpha). \Ind (y,x) 
					\tag{\sym} \\
					&\alphaequiv \forall x \colon K (\beta) \exists y \colon K(\alpha). \Ind (x,y)\\
\end{align*}
By an exactly analogous argument we can establish $(\tran_{\simeq})$.

$(\eqr^n_{\Ind}) \Rightarrow (\eqr^{n+1}_{\Ind})$: Let $K$ be a sort with $l(K)=n+1$. 
Firstly, note that it suffices to establish $(\eqr_{\Ind_R^p})$ for any $R>K$ and $p \colon R \rightarrow K$ since $\Ind$ is a conjunction over all such $R$ and $p$.
Thus, the properties to be shown are the following:

\begin{equation}
\tag{$\refl_{\Ind_R^p}$} \forall x \colon K. \Ind_R^p (x,x)
\end{equation}
\begin{equation}
\tag{$\sym_{\Ind_R^p}$} \forall x,y \colon K. \Ind_R^p (x,y) \Rightarrow \Ind_R^p (y,x)
\end{equation}
\begin{equation}
\tag{$\tran_{\Ind_R^p}$} \forall x,y,z \colon K. \Ind_R^p (x,y) \wedge \Ind_R^p (y,z) \Rightarrow \Ind_R^p (y,z)
\end{equation}

We take each of the three statements to be shown in turn. Firstly, for $(\refl_{\Ind_R^p})$ we have:
\begin{align*}
\Ind_R^p (x,x) &\eqdef \bigwedge \lbrace \forall \Gamma_{\alpha,\beta} R(\alpha) \simeq R(\beta) \: \vert \: \alpha, \beta \colon R, \alpha_p = x, \beta_p = x, \forall q \neq p( \alpha_q=\beta_q) \rbrace / \conteq \\
&= \:\:\:\bigwedge \lbrace \forall \Gamma_{\alpha,\alpha} R(\alpha) \simeq R(\alpha) \: \vert \: \alpha \colon R, \alpha_p = x \rbrace / \conteq \\
\end{align*}
where the equality follows since $\alpha$ and $\beta$ in the general definition of $\Ind_R^p (x,y)$ are stipulated to differ only in the variable they assign to position $p$, which is here going to be the same variable. 
We therefore get that $\Ind_R^p(x,x)$ holds iff $R(\alpha) \simeq R(\alpha)$ holds, and we know that the latter is the case from the inductive hypothesis together with the first part of the inductive step above, since $R>K$.

Secondly, for $(\sym_{\Ind_R^p})$ we have:
\begin{align*}
\Ind_R^p (x,y) &\eqdef \bigwedge \lbrace \forall \Gamma_{\alpha,\beta} R(\alpha) \simeq R(\beta) \: \vert \: \alpha, \beta \colon R, \alpha_p = x, \beta_p = y, \forall p \neq q(\alpha_q=\beta_q)\rbrace / \conteq \\
\\
&\:\Rightarrow  \bigwedge \lbrace \forall \Gamma_{\alpha,\beta} R(\beta) \simeq R(\alpha) \: \vert \: \alpha, \beta \colon R, \alpha_p = x, \beta_p = y, \forall p \neq q(\alpha_q=\beta_q) \rbrace / \conteq \tag{\sym$_{\simeq}$}\\
&\alphaequiv \bigwedge \lbrace \forall \Gamma_{\alpha,\beta} R(\beta) \simeq R(\alpha) \: \vert \: \beta, \alpha \colon R, \beta_p = x, \alpha_p = y, \forall p \neq q(\alpha_q=\beta_q) \rbrace / \conteq \\
&\eqdef \Ind_R^p (y,x)
\end{align*}

Thirdly, for $(\tran_{\Ind_R^p})$ we have:
\begin{align*}
\Ind_R^p (x,y) &\wedge \Ind_R^p (y,z) \\
&\eqdef \bigwedge \lbrace \forall \Gamma_{\alpha,\beta} R(\alpha) \simeq R(\beta) \: \vert \: \alpha, \beta \colon R, \alpha_p = x, \beta_p = y, \forall p \neq q(\alpha_q=\beta_q)\rbrace / \conteq \\
&\:\:\:\:\wedge \bigwedge \lbrace \forall \Gamma_{\alpha,\beta} R(\alpha) \simeq R(\beta) \: \vert \: \alpha, \beta \colon R, \alpha_p = y, \beta_p = z, \forall p \neq q(\alpha_q=\beta_q)\rbrace / \conteq \\
&\Leftrightarrow \bigwedge \lbrace \forall \Gamma_{\alpha,\beta} \cup \Gamma_{\beta,\gamma} R(\alpha) \simeq R(\beta) \wedge R(\beta) \simeq R(\gamma) \: \vert \: \alpha_p = x, \beta_p = y, \gamma_p=z \rbrace / \conteq \\
&\Rightarrow \bigwedge \lbrace \forall \Gamma_{\alpha,\gamma} R(\alpha) \simeq R(\gamma) \: \vert \: \alpha_p = x, \gamma_p = z)\rbrace / \conteq \\
&\eqdef \Ind_R^p (x,z)
\end{align*}
This establishes the fact that $\Ind(x,y)$ is an equivalence relation, thus completing the inductive step.
\end{proof}

We conclude this section by comparing our syntactic definitions to the usual notions.
In $\Lcat$, for $x,y \colon O$ let $\text{Iso} (x,y)$ denote the formula expressing that $x$ and $y$ are isomorphic in the traditional sense (``There exist mutually inverse $f \colon A(x,y)$ and $g \colon A(y,x)$'') 
Call this notion \emph{categorical isomorphism}. The following proposition days that in $\Tcat$ the two notions of isomorphism are logically equivalent.

\begin{prop}\label{catspec}
$\Tcat \models \text{\emph{Iso}} (x,y) \Leftrightarrow x \cong y$
\end{prop}
\begin{proof}
Firstly, note that by the axioms for equality on $A$ we will get that for any $f,f' \colon A(x,y)$
\[
\Tcat \models f \cong_A f' \Leftrightarrow f=_Af'
\]
This allows us to replace (up to logical equivalence) the $=_A$ as it appears in $\text{Iso} (x,y)$ with our notion of $\cong_A$. We will assume we have done so.
The ($\Rightarrow$) direction follows from the well-known result (proven e.g. in \cite{MFOLDS}) that two categorically isomorphic objects in a category satisfy exactly the same $\Lcat$-formulas, and are therefore indistinguishable in exactly the sense of $x \cong y$.
For the ($\Leftarrow$) direction, note that from the Yoneda Lemma we know that $x$ and $y$ are categorically isomorphic if and only if their images under the Yoneda embedding are isomorphic, i.e. if we have a natural isomorphism from $\yoneda (-,x)$ to  $\yoneda (-,y)$. The existence of such a natural isomorphism can be expressed as the following $\Lcat$-formula
\begin{align*}
\Yso (x,y) \eqdef \:&\forall z,w \colon O \forall f \colon A(w,z) \forall g \colon A(z,x) \forall h \colon A(w,x) \forall k \colon A(w,y) \forall l \colon A (z,y) \\
&A(z,x) \simeq A(z,y) \wedge (\circ(g,f,h) \wedge \Ind(h,k) \leftrightarrow \Ind(g,l) \wedge \circ (l,f,k)))
\end{align*}
By the Yoneda Lemma we have $\Tcat \models \forall x \forall y \Iso (x,y) \Leftrightarrow \Yso (x,y)$. But $\Yso (x,y)$ is clearly a consequence of $x \cong y$ since both conjuncts appear in the definitions given above if, using Proposition \ref{degequiv}, we regard $\leftrightarrow$ as an instance of $\simeq$. We thus get $x \cong y \Rightarrow \Iso(x,y)$
\end{proof}

\begin{remark}\label{wit}
An important difference with our general notion $\cong$ and the usual notions of isomorphim is that $x \cong y$ is not necessarily ``witnessed'' by some other variable, e.g. an arrow $f$ as in $\text{Iso} (x,y)$. 
Thus, at this level, it is not possible to define a \emph{set} of isomorphisms $x \cong y$, but rather only a \emph{proposition} expressing whether or not $x$ and $y$ are isomorphic. To overcome this limitation, in Section \ref{secsat} we will interpret the existential quantifiers as untruncated $\Sigma$-types, which will recover in the semantics the intuitively correct $h$-level.
\end{remark}


\section{Saturation in Type Theory}\label{secsat}


We now interpret our syntactic definition of isomorphism in type theory.
However, in order to interpret $x \cong y$ as defined in Section \ref{syndef} we make one crucial change to the usual interpretation, which is that we will not truncate the existential quantifiers that appear in the definition of $\simeq$ in Definition \ref{equivalence} and interpret them instead as full $\Sigma$-types.
This is made precise in the following definition.

\begin{defin}[Interpretation of $\cong$ in type theory]\label{interpretationofequivalence}
Let $\L$ be a FOLDS signature and $\M$ an $\L$-structure. 
For $K \in \obL$ and $\alpha, \beta \colon \MK$ we define the following type (in context $(\partial \alpha \cup \partial \beta)^\M$):
\begin{align*}
K(\alpha) \simeq^\M K (\beta) &\eqdef \bigg[ \Pitype{x \colon \MK(\alpha)} \Sigmatype{\:\:y \colon \MK(\beta)} \Ind^\M (x,y) \times \bigg( \Pitype{y' \colon \MK(\beta)} \Ind^\M (x,y') \rightarrow y \cong^\M y'\bigg) \bigg] \times \\
& \:\:\:\:\:\:\:\:\: \bigg[ \Pitype{\begin{subarray} \:x,x' \colon \MK(\alpha) \\ y,y' \colon \MK(\beta)\end{subarray}} \big( \Ind^\M (x,y) \times \Ind^\M (x',y') \times y \cong^\M y' \rightarrow x \cong^\M x' \big) \bigg] \times \\
& \:\:\:\:\:\:\:\:\: \bigg[ \Pitype{y \colon K (\beta)} \Sigmatype{x \colon \MK(\alpha)} \Ind^\M(x,y) \bigg]\\
\end{align*}
By mutual induction we also define for $x,y \colon \MK$ the following type (in context $(\dep(x) \cup \dep(y))^\M$):
\begin{align*}
\Ind^\M (x,y) \eqdef \underset{\begin{subarray} \:R \in \L_x \cap \L_y \\ \:p \in \L(R, K) \end{subarray}} {\times} \lbrace \Pitype{\mathbf{z} \colon \Gamma^\M_{\alpha,\beta}} R(\alpha) \simeq^\M R(\beta) \: \vert \: \alpha_p = x, \beta_p = y, \forall q \neq p(\alpha_q=\beta_q)\rbrace / \conteq
\end{align*}
Whenever $x,y \colon \MK$ are such that $\partial x = \partial y$ we write $x \cong^\M y$ for $\Ind^\M (x,y)$.
Thus, given any FOLDS signature $\L$, any $K \in \obL$ we obtain a relation (in an appropriate context): $$\cong_{\MK} \colon \MK \times \MK \rightarrow \mathcal{U}$$
\end{defin}


\begin{notation}
Whenever the sort of $x$ and $y$ is understood we will usually denote the type $\cong_{\MK} (x,y)$ as simply $x \cong y$.
\end{notation}


Let $\M$ be an $\L$-structure and $K$ any sort in $\L$, with $\MK$ its interpretation.
Proposition \ref{reflprop} holds also in type theory (since the proof was purely formal) and so the relation $\cong_{\MK}$ will be reflexive.
In particular, this means that given any context $\Gamma$ appropriate to $K$ we have
\[
\Gamma^\M, x \colon \MK \vdash \rho_x \colon x \cong x
\]
where $\rho_x$ is some term witnessing the reflexivity.
This gives us a canonical map
\[
\Gamma^\M, x,y \colon \MK \vdash \mathtt{idtoiso}_{x,y} \colon x= y \rightarrow x \cong y
\]
by the usual induction on identity. 

\begin{defin}[Saturation]\label{saturation}
Let $\L$ be a FOLDS signature, $K \in \obL$, $\Gamma$ a context appropriate to $K$ and $\M$ an $\L$-structure. Then the \textbf{saturation for $\MK$} is the following proposition (in context $\Gamma^\M$):
\[
\SatGK \eqdef \underset{x, y \colon \MK}{\Pi} \mathtt{isequiv} (\mathtt{idtoiso}_{x,y})
\]
For an $\L$-structure $\M$, we say that $K$ is \textbf{saturated} (or that $\M$ is \textbf{$K$-saturated}) if $\SatGK$ is inhabited.
We will say that an $\L$-structure is \textbf{saturated at level $n$} (or \textbf{$n$-saturated}) if it is $K$-saturated for any $K \in \obL$ with $l(K) \leq n$.
We will also say that an $\L$-structure $\M$ is \textbf{totally saturated} if it is saturated at every level. 
We write $\SStrucL$ (resp. $\SStrucL^n$) for the type of totally saturated (resp. $n$-saturated) $\L$-structures and $\SMod{\T}{}$ (resp. $\SMod{\T}{n}$) for the type of totally saturated (resp. $n$-saturated) $\T$-models (for some $\L$-theory $\T$). 
\end{defin}

Based on Definition \ref{saturation} we regard $\Struc{\L}$ (resp. $\Mod{\T}$) as the type of ``totally unsaturated'' types of $\L$-structures (resp. models of $\T$). The following example illustrates the idea.

\begin{exam}\label{Lrgsatexample}
\begin{align*}
\Struc{\Lrg} &\eqdef \Sigmatype{\begin{subarray} \:O \colon \mathcal{U} \\ A \colon O \times O \rightarrow \mathcal{U} \end{subarray}} \: \big( \Pitype{x \colon A} A(x,x) \rightarrow \mathcal{U} \big)
\\
\SStruc{\Lrg}{1} &\simeq \Sigmatype{\begin{subarray} \:O \colon \mathcal{U} \\ A \colon O \times O \rightarrow \mathcal{U} \end{subarray}} \: \Pitype{x \colon A} \big(A(x,x) \rightarrow \textbf{Prop}_{\mathcal{U}}\big)
\end{align*}
It is important to note the difference between $\eqdef$ and $\simeq$ above. 
The RHS in the first line is the definition of $\Struc{\Lrg}$ whereas the RHS of the second line is a consequence of Proposition \ref{degiso}.
\end{exam}

The most relevant notion for practical purposes is saturation at a particular level (or total saturation). This is because $K$-saturation for an individual $K$ may be of very little use unless we know that everything that depends on $K$ is already saturated. 
For example, it is not of much use to state the saturation condition for objects in a precategory unless we already know that the hom-sets are $h$-sets.
On the contrary, if we know that $\L$ is saturated at level $n$ and $K$ is a sort of level $n$, then $K$-saturation means that isomorphism in $K$ is expressed in terms of the identity types of sorts of lower level, since these have already been declared equivalent to isomorphism.
This has the consequence that saturation at a particular level correlates very well with $h$-levels as the following proposition establishes.

\def\N{\mathcal{N}}

\begin{prop}\label{levhlev}
Let $\L$ be a FOLDS signature and $\M$ an $\L$-structure saturated at level $n$. Then for any $K \in \obL$ with $l(K)=n$, $\MK$ is of $h$-level $n$.
\end{prop}
\begin{proof}
We proceed by induction on the level of $K$.
For the case $l(K)=1$ note that $\M$ is $1$-saturated if and only if for any sort $K$ of level $1$, the map $$\idtoiso_{x,y} \colon x =_{\MK} y \rightarrow (x \cong y)^\M$$ is an equivalence.
But by Proposition \ref{degiso} we have $(x \cong^\M y) \alphaequiv \top$ and hence $\mathtt{idtoiso_{x,y}}$ is the unique map into \textbf{1}.
Therefore, we have
\[
\SatGK \simeq \underset{x,y \colon \MK}{\Pi} x =_{\MK} y 
\]
and the RHS is equivalent to $\mathtt{isprop}(\MK)$. Therefore, $K$ is saturated in $\M$ if and only if $\MK$ is of $h$-level 1.
Now assume that the proposition holds for $n$ and let $K$ be a sort of level $n+1$ and $\M$ an $\L$-structure saturated at level $n+1$. Saturation for $K$ asserts the equivalence between the identity type on $K$ with $x\cong_K y$. 
But $x \cong y$, as a type dependent on $\MK$, consists of the following (non-dependent) product: 
\[
\underset{R \in \L_x \cap \L_y}{\times} \: \: \underset{p \colon R \rightarrow K}{\times} (\Ind_R^p (x,y))^\M
\]
But $\Ind_R^p$ is itself a non-dependent product of $\Pi$-types of the form
\[
\underset{\alpha,\beta}{\Pi} \big[(R(\alpha) \simeq^\M R(\beta)\big]
\]
Since $\MR(\beta)$ and $\MR(\alpha)$ are of $h$-level $n$ by the inductive hypothesis, then by the explicit form of $R(\alpha) \simeq^\M R(\beta)$  (cf. Definition \ref{interpretationofequivalence}) we also get that $R(\alpha) \simeq^\M R(\beta)$ is of $h$-level $n$.
Since (dependent or non-dependent) products preserve $h$-level we get that $(x \cong y)^\M$ is of $h$-level $n$. Thus, if $\MK$ is saturated, its identity types will be of $h$-level $n$, and thus $\MK$ will itself be of $h$-level $n+1$, as required.
\end{proof}

The following lemma is also an easy consequence of our definitions. It makes precise in what way the syntactically defined notion of equivalence becomes a ``functional'' notion in the semantics, as long as everything ``above'' has been saturated.

\begin{lemma}\label{indtofun}
Let $\L$ be a FOLDS signature and $\M$ an $\L$-structure saturated at level $n$. Then for any $K \in \obL$ with $l(K) = n$ we have:
\[
\big[ K(\alpha) \simeq^\M K (\beta) \big] \simeq \Sigmatype{f \colon \MK(\alpha) \rightarrow \MK(\beta)} \isbijection(f) \times \Pitype{x \colon \MK(\alpha)} \emph{\Ind}^\M (x,f(x))
\]
\end{lemma}
\begin{proof}
By inspecting the definition of we can see that the first type in $K(\alpha) \simeq^\M K (\beta)$ is equivalent to the type of functions $\MK(\alpha) \rightarrow \MK(\beta)$, since $\cong^\M$ can simply be replaced by $=_{\MK}$ since we are assuming $\M$ is $n$-saturated.
And the last two types state exactly that this function is injective and (split) surjective, and hence bijective.
\end{proof}




\begin{remark}\label{nonelem}
Saturation as defined by Definition \ref{saturation} relies essentially on the type-theoretic semantics of FOLDS.
This means that saturated models of $\T$ are defined ``semantically'' since we have not provided a way of expressing the condition of saturation in the syntax of FOLDS. One might therefore wonder if there is a way to define saturation in an entirely syntactic manner in FOLDS, and then obtain the usual saturated notions by interpreting that definition. This can indeed be done as long as we syntactically define ``transport bijections'' that for each ``inhabitant'' of $x \cong_K y$ yield an ``equivalence'' between $A(x,\mathbf{w})$ and $A(y,\mathbf{w})$ for any $A$ dependent on $K$. We do this in full generality in \cite{TsemHMTII} for the purposes of a completeness proof, but this machinery is somewhat independent of the purposes of this paper.
\end{remark}

\section{A Higher Structure Identity Principle}\label{HSIPsec}

A Structure Identity Principle in the sense of \cite{HTT} asserts something along the lines of: ``Isomorphism between structures of a certain kind is equivalent to identity.''
In \cite{HTT} such a statement was proven for set-level structures, i.e. for first-order structures in the traditional sense.
Here we prove a generalized Structure Identity Principle where by ``structure'' we will understand saturated $\L$-structures for (finite) FOLDS signatures $\L$ of height $3$ and by ``isomorphic'' we will understand the pre-existing notion of FOLDS equivalence, suitably internalized in type theory.

Let us first define FOLDS equivalence. Fix an $\L$ and let us for the time being work ``extrenally'' i.e. in a traditional set-theoretic set-up, writing $M, N,\dots$ instead of $\M,\N,\dots$ to distinguish the two notions of $\L$-structure.
An $\L$-structure is a functor $M \colon \L \rightarrow \textbf{Set}$ and a \emph{homomorphism} of $\L$-structures is a natural transformation between the corresponding functors. 
Write $y$ for the Yoneda embedding and $\hat{y}$ for subfunctor of $y$ that misses the identity on the given object.
For $K \in \obL$ a \emph{$K$-boundary} of an $\L$-structure $M$ is a natural transformation $\delta \colon \hat{y}K\Rightarrow M$ 
and we define the \emph{fiber} of $M(K)$ over $\delta$ as the set
$
M(K)[\delta] = \lbrace a \in M(K) \colon \partial a = \delta \rbrace
$. Now let $M,N$ be $\L$-structures and let $f \colon M \rightarrow N$ be an $\L$-structure homomorphism and let $\delta$ be a $K$-boundary of $M$. This induces a $K$-boundary $f \circ \delta$ of $N$ by composition with $f$ and thereby induces a map $f_\delta$ on the fibers that takes
$a \mapsto f_K(a)$.
(The fact that this is well-defined follows from the naturality of $f$.)
%
We now say that $f$ is \emph{fiberwise surjective} if and only if for all $K$ in $\L$ and all $K$-boundaries $\delta$ the map $m_\delta \colon M(K)[\delta] \rightarrow M(K)[m \circ \delta]$ is surjective. 

\begin{defin}[\cite{MFOLDS}]
We say that $M$ and $N$ are \emph{FOLDS $\L$-equivalent} if there exists an $\L$-structure $P$ and fiberwise surjective homomorphisms $m \colon P \rightarrow M$ and $n \colon P \rightarrow N$.
\end{defin}

We now internalize this definition in type theory.
To even write it down, however, we will now restrict ourselves to signatures $\L$ of height $3$ and we will assume that $\StrucL$ is given by the $\Sigma$-type where the relations in $\L$ have been ``written in''. To be more precise, we will adopt the following notational conventions. For the rest of the paper $\L$ is a FOLDS signature of height $3$.

\begin{notation}
We denote the objects in $\L$ of level $3$ by $\mathbf{O} = O_1, \dots , O_n$, those of level $2$ by $\mathbf{A} = A_1, \dots A_m$ and those of level $1$ by $\mathbf{R} = R_1, \dots, R_k$. 
We will abbreviate a sequence of variable declarations $x_1 \colon O_1, \dots, x_n \colon O_n$ as $\mathbf{x} \colon \mathbf{O}$, leaving the indexing implicit. Similarly, we will abbreviate a product $A_1 \rightarrow B_1 \times \dots \times A_n \rightarrow B_n$ as $\mathbf{A} \rightarrow \mathbf{B}$.
In general, whenever there is any boldface character in our definitions below, this will mean that the definition is taken over the whole (non-dependent) product that the implicit indexing suggests.
For example, an expression of the form $\Pitype{\mathbf{x} \colon \mathbf{O}} T(m_{\mathbf{O}} (\mathbf{x}))$ is an abbreviation for
\[
\Pitype{x_1 \colon O_1, \dots x_n \colon O_n} \bigg[ T(m_{O_1} (x_1,\dots,x_n)) \times \dots \times T(m_{O_n} (x_1,\dots,x_n)) \bigg]
\]
\end{notation}


\begin{defin}[$\L$-homomorphism]\label{HomL}
Let $\M$ and $\N$ be $\L$-structures. The type of $\L\textbf{-homomorphisms}$ between $\M$ and $\N$ is defined as
\[
\HomL (\M, \N) \eqdef \underset{\begin{subarray} \:m_\mathbf{O} \colon \mathbf{O}^\M \rightarrow \mathbf{O}^\N \\ m_\mathbf{A} \colon \underset{\mathbf{x} \colon \mathbf{O{^\M}}}{\Pi} \mathbf{A}^\M(\mathbf{x}) \rightarrow \mathbf{A}^{\N}(\mathbf{m}_{\mathbf{O}}(\mathbf{x})) \end{subarray}}{\Sigma} \:\:\underset{\begin{subarray} \:\mathbf{x} \colon \mathbf{O}^\M \\ \mathbf{y} \colon \mathbf{A}^\M(\mathbf{y}) \end{subarray}}{\Pi} \mathbf{R}^\M (\mathbf{y}) \rightarrow \mathbf{R}^{\N}(m_\mathbf{O}(\mathbf{x}) (\mathbf{y}))
\]
As suggested by the notation above, for any $m \colon \HomL (\M, \N)$ and $K \in \obL$ we will write $m_K (\textbf{x})$ for the component of $m$ corresponding to sort $K$ in some choice of variables $\mathbf{x}$.
\end{defin}

%

\begin{defin}[Split Surjective Map]\label{rfibsurjK}
Let $f \colon A \rightarrow B$ be any map. Then $f$ is a \textbf{split surjection} if the following type is inhabited
\[
\isrfibsurjK (f) \eqdef \Pitype{y \colon B} \:\: \Sigmatype{x \colon A} f (x) = y
\]
\end{defin}

\begin{remark}
The type $\isrfibsurjK (f)$ is not a proposition, and this is necessary. We must therefore think of $\isrfibsurjK(f)$ as the type of \emph{splittings} of $f$ and there could be many distinct such splittings.
\end{remark}

\begin{defin}\label{rfibsurj}
Let $m \colon \HomL (\M,\N)$. We define
\[
\isrfibsurj (m) \eqdef \isrfibsurjK (m_{\mathbf{O}}) \times \Pitype{\mathbf{x} \colon \mathbf{O}} \isrfibsurjK (m_{\mathbf{A}} (\mathbf{x})) \times \Pitype{\begin{subarray} \:\mathbf{x} \colon \mathbf{O} \\ \mathbf{y} \colon \mathbf{A}(\mathbf{x}) \end{subarray}} \isrfibsurjK (m_{\mathbf{R}}(\mathbf{x}, \mathbf{y}))
\]
We define the type of \textbf{fiberwise surjective maps} from $\M$ to $\N$ as $$\FibSurj (\M, \N) \eqdef \underset{m \colon \HomL (\M, \N)}{\Sigma} \isrfibsurj(m)$$
We call a term $m \colon \FibSurj (\M, \N)$ a \textbf{fiberwise surjection}. 
\end{defin}

\begin{remark}
Once again, we need to make clear here that $\FibSurj (\M,\N)$ is not a $\Sigma$-type comprised of an $\L$-homomorphism that satisfies a property, since $\mathtt{fibsurj}(f)$ will not in general be a proposition. One must rather think of fiberwise surjective maps as $\L$-homomrphisms together with a \emph{choice} of a splitting that witnesses the fact that it is surjective.
\end{remark}

\begin{notation}
As usual we will abuse notation and treat $m \colon \FibSurj (\M, \N)$ as if it were comprised purely of its ``functional'' part. We will thus describe the action of $m$ on terms of $\M$ simply by appending $m$ to the given terms, suppressing subscripts. For example, we will write $mx$ for the result of applying $m_{O}$ to some $x \colon O^\M$ and similarly we will write $mf$ for the result of applying $m_A (\mathbf{x})$ to some $f \colon A^\M (\mathbf{x})$.  When it is required to make the particular splitting with which some $m \colon \FibSurj (\M, \N)$ comes equipped we will refer to it by $s$ and describe its action in the same abbreviated manner as we do with $m$.
\end{notation}

\begin{defin}[Type-Theoretic FOLDS Equivalence]\label{reedyequiv}
For any FOLDS signature $\L$ (of height $3$) and any $\M,\N \colon \StrucL$ we define 
\[
\M \simeq_{\L} \N \eqdef \Sigmatype{\mathcal{P} \colon \StrucL} \FibSurj (\mathcal{P},\M) \times \FibSurj (\mathcal{P}, \N)
\]
Exactly analogously we define $\simeq_\L$ on 
$\Mod{\T}$, $\SStrucL$ and $\SMod{T}{}$ for any $\L$-theory $\T$.
\end{defin}




Definition \ref{reedyequiv} thus gives us a relation $\simeq_\L \colon \SMod{\T}{} \times \SMod{\T}{} \rightarrow \mathcal{U}$.
All the components are now in place for us to prove our higher structure identity principle.
To do so, we first require the following lemma.

\begin{lemma}\label{fibsurjpres2}
Let $\L$ be a FOLDS signature with $\heightof (\L)=3$, $\M$ and $\N$ be saturated $\L$-structures, and $m \colon \emph{\FibSurj} (\M, \N)$ a fiberwise surjective $\L$-homomorphism. 
\begin{enumerate}
\item For any $f \colon A(\mathbf{x}), g \colon A (\mathbf{y})$ we have $\emph{\Ind}^\N (mf,mg) \leftrightarrow \emph{\Ind}^\M(f,g)$. In particular, we have $mf \cong mg \leftrightarrow f \cong g$.
\item For any $x,y \colon O$ we have $(mx \cong^\M my) \simeq (x \cong^\N y)$.
\end{enumerate}
\end{lemma}
\begin{proof}
For (1) we have that $\Ind^\M (f,g)$ will be of the following form
\[
\underset{\mathbf{x} \colon \mathbf{O}^\M}{\Pi} \mathbf{R}^\M (\mathbf{x},\mathbf{f}) \leftrightarrow \mathbf{R}^\M (\mathbf{y}, \mathbf{g})
\]
and similarly $\Ind^\N (mf,mg)$ will be of the following form
\[
\underset{\mathbf{x} \colon \mathbf{O}^\N}{\Pi} \mathbf{R}^\N (\mathbf{x},m\mathbf{f}) \leftrightarrow \mathbf{R}^\N (\mathbf{y}, m\mathbf{g})
\]
Since each $\mathbf{R}$ in the above types will be propositions so will the whole of $\Ind^\M (f,g)$ and $\Ind^\N (f,g)$. Therefore, it suffices to prove that they are logically equivalent.
To that end define the maps
\[
\xymatrix{
(\underset{\mathbf{x} \colon \mathbf{O}^\M}{\Pi} \mathbf{R}^\M (\mathbf{x},\mathbf{f}) \leftrightarrow \mathbf{R}^\M (\mathbf{y}, \mathbf{g})) \ar[r]_{\sigma}  &(\underset{\mathbf{x} \colon \mathbf{O}^\N}{\Pi} \mathbf{R}^\N (\mathbf{x}, m\mathbf{f}) \leftrightarrow \mathbf{R}^\N (\mathbf{y}, m\mathbf{g})) \\
\eta \ar@{|->}[r]  &\lambda \mathbf{x}. (\lambda \mathbf{r}. m(p_1(\eta_{s(\mathbf{x})}))(s\mathbf{r}), \lambda \mathbf{r}. m(p_2(\eta_{s(\mathbf{x})}))(s\mathbf{r})) \\
}
\]
and
\[
\xymatrix{
(\underset{\mathbf{x} \colon \mathbf{O}^\M}{\Pi} \mathbf{R}^\M (\mathbf{x},\mathbf{f}) \leftrightarrow \mathbf{R}^\M (\mathbf{y}, \mathbf{g}))   &(\underset{\mathbf{x} \colon \mathbf{O}^\N}{\Pi} \mathbf{R}^\N (\mathbf{x}, m\mathbf{f}) \leftrightarrow \mathbf{R}^\N (\mathbf{y}, m\mathbf{g})) \ar[l]_{\tau} \\
\lambda \mathbf{x}.(\lambda \mathbf{r}. s(p_1(\epsilon_{\mathbf{x}}) (m\mathbf{r})), \lambda \mathbf{r}. s(p_2(\epsilon_{\mathbf{x}}) (m\mathbf{r})))    & \epsilon \ar@{|->}[l]
}
\]
The fact that $mf \cong mg \leftrightarrow f \cong g$ follows by as the special case when $\mathbf{x} \syneq \mathbf{y}$. 

For (2), we define maps as shown:
\[
\hspace{-1cm}\xymatrix{
x \cong y \ar@/_/[r]_{e} \ar@{=}[d]  &mx \cong my \ar@/_/[l]_{d} \ar@{=}[d] \\
\Pitype{\mathbf{z} \colon \mathbf{O}^\M} \mathbf{A}^\M (x,\mathbf{z}) \simeq \mathbf{A}^\M (y,\mathbf{z})  \ar@/_/[r]_{e} \ar@{=}[d]  &\Pitype{\mathbf{z} \colon \mathbf{O}^\N} \mathbf{A}^\N (mx,\mathbf{z}) \simeq \mathbf{A}^\N (my,\mathbf{z})  \ar@/_/[l]_{d} \ar@{=}[d]\\
\Pitype{\begin{subarray} \:\mathbf{z} \colon \mathbf{O}^\M \\ f \colon A^\M(x, \mathbf{z}) \end{subarray}} \Sigmatype{g \colon A^\M(y, \mathbf{z})} \Ind^\M (f,g) \ar@/_/[r]_{e}   &\Pitype{\begin{subarray} \:\mathbf{z} \colon \mathbf{O}^\N \\ f \colon A^\N(mx, \mathbf{z}) \end{subarray}} \Sigmatype{g \colon A^\N(my, \mathbf{z})} \Ind^\N (mf,mg) \ar@/_/[l]_{d}  \\
\eta \ar@{|->}[r]  &\lambda \mathbf{z}\lambda f. \bigg\langle m(p_1(\eta_{s\mathbf{z}, sf})),  \sigma(p_2(\eta_{s\mathbf{z}, sf}))  \bigg\rangle  \\
\lambda \mathbf{z}\lambda f. \bigg\langle s(p_1(\epsilon_{m\mathbf{z},mf})), \tau (p_2(\epsilon_{m\mathbf{z},mf}))  \bigg\rangle 			  &\epsilon \ar@{|->}[l]
}
\]
It is easy to check that $d$ and $e$ are inverses to each other, thus establishing the required equivalence.
\end{proof}



We are now ready to state and prove the main result of this paper which is a fully precise version of the guiding Pre-Theorem \ref{satpretheo} in the Introduction.

\begin{theo}[Higher Structure Identity Principle]\label{HSIP}
For any $\L$ with $\heightof (\L)=3$, any $\L$-theory $\T$ and 
for any saturated $\T$-models $\M$ and $\N$ of an $\L$-theory $\T$
we have 
\[
(\M \simeq_\L \N) \simeq (\M =_{\SMode{\T}{}} \N)
\]
\end{theo}
\begin{proof}
Let $\M$ and $\N$ be saturated models for some $\L$-theory $\T$.
Clearly, it suffices to prove that $\FibSurj (\M,\N) \simeq \M =_{\SMod{\T}{}} \N$.
We define a map $$k \colon \FibSurj (\M,\N) \rightarrow \M =_{\SMod{\T}{}} \N$$ 
Let $m \colon \FibSurj (\M,\N)$. 
By univalence, to produce a term of type $\M =_{\SMod{\T}{}} \N$ it suffices to produce component-wise equivalences for all the defining data of $\M$ and $\N$. Since $m$ is assumed fiberwise surjective we know that it provides component-wise surjections between the defining data of $\M$ and $\N$. 
Thus, it suffices to prove that $m$ is also component-wise an injection. We do so by taking the data of $\L$ level-by-level, abusing notation by ignoring the necessary transfers. 

Since $\M$ and $\N$ are saturated, by Proposition \ref{levhlev} each $\mathbf{R}^\M (\mathbf{x})$ and $\mathbf{R}^\N (\mathbf{y})$ will be of $h$-level 1. Therefore, all the
\[
m_{\mathbf{R}} (\mathbf{x}) \colon \mathbf{R}^\M (\mathbf{x}) \rightarrow \mathbf{R}^\N (\mathbf{y})
\] 
are trivially also injections, and therefore equivalences. 

Now consider $m_{\mathbf{A}}(\mathbf{x}) \colon A^\M (\mathbf{x}) \rightarrow A^{\N}(m\mathbf{x})$. By Proposition \ref{levhlev} once again we know that $m_{\mathbf{A}}(\mathbf{x})$ is a map of $h$-sets and therefore to show that it is an equivalence it suffices to show $m_{\mathbf{A}}(\mathbf{x}) (\alpha) = m_{\mathbf{A}}(\mathbf{x})  (\beta) \rightarrow \alpha = \beta$. But then note that we have the following composite map
\[
m_{\mathbf{A}}(\mathbf{x}) (\alpha) = m_{\mathbf{A}}(\mathbf{x})  (\beta) \overset{(1)}{\longrightarrow} m_{\mathbf{A}}(\mathbf{x}) (\alpha) \cong m_{\mathbf{A}}(\mathbf{x})  (\beta) \overset{(2)}{\longrightarrow} \alpha \cong \beta \overset{(3)}{\longrightarrow} \alpha = \beta
\]
where (1) is by induction on identity, (2) is by the first part of Lemma \ref{fibsurjpres2} since $m$ is assumed fiberwise surjective and (3) is because $\M$ is assumed saturated.

Finally, for $m_{\mathbf{O}} \colon \mathbf{O}^\M \rightarrow \mathbf{O}^\N$, to show that it is injective it suffices to show that there is an equivalence $(x=y) \simeq (mx = my)$ for any $x,y \colon \mathbf{O}$. Since both $\M$ and $\N$ are saturated this is equivalent to showing that $(x \cong y) \simeq (mx \cong my)$. But since $m$ is fiberwise surjective, this follows from the second part of Lemma \ref{fibsurjpres2}. 

This completes the definition of the map $k$. Now we need to show that $k$ is itself an equivalence. To see this note that by univalence any $p \colon \M =_{\SMod{\T}{}} \N$ gives rise to a fiberwise surjective homomorphism $m_p$ with components given by the equivalences that $p$ induces. This gives a map 
\[
l \colon \M =_{\SMod{\T}{}} \N \rightarrow \FibSurj (\M, \N)
\]
which can easily be seen to be a two-sided inverse for $k$. Thus, $k$ is an equivalence, and this completes the proof.
\end{proof}

In order to test the adequacy of our Structure Identity Principle we need to show that it specializes to familiar cases. In this case, we need to show that Theorem 9.4.16 of \cite{HTT} follows as a special case of our Theorem \ref{HSIP}.
The next series of Propositions and Lemmas aim exactly at demonstrating this.

\begin{prop}\label{precatsat2}
The type ${\SMode{\Tcat}{2}}$ of $2$-saturated models of the $\Lcat$-theory of categories is equivalent to the type $\textbf{\emph{PreCat}}$ of precategories.
\end{prop}
\begin{proof}
By the axioms for $=_A$ in the theory of categories we can obtain that
\[
\Tcat \models \forall x,y \colon O \forall f,g \colon A(x,y). f \cong g \leftrightarrow f=_A g
\]
So if $\C$ is any $2$-saturated $\Tcat$-model we will have 
\[
\Pitype{x,y \colon O^\C} \Pitype{f,g \colon A^\C (x,y)} (f=_Ag)^\C \leftrightarrow f=_{A^\C(x,y)} g
\]
And since $(f=_Ag)^\C$ will be of $h$-level 1 (by Proposition \ref{levhlev}) this means exactly that $A^\C (x,y)$ is a set for all $x,y \colon O^\C$. 
What remains to be shown is that the (interpretations of) the axioms of $\Tcat$ correspond exactly to the axioms for precategories (but in relational form).
This is straightforward and already established in \cite{TsemHMT} as Proposition 7.1 and has even been formalized in \cite{UniMath}.
\end{proof}

\begin{lemma}\label{yonlemlem}
Let $\C$ be a precategory. For every $x,y,z \colon O^\C$, $\alpha \colon A^\C (z,x) \rightarrow A^\C (z,y)$ and $h \colon A^\C (z,x)$ we have
\[
\emph{\Ind}^\C (h, \alpha(h)) \simeq \Pitype{\begin{subarray}\: w \colon O^\C \\ g \colon A^\C(w,z)  \end{subarray}} \alpha(h) \circ g = \alpha (h \circ g)
\]
\end{lemma}
\begin{proof}
Since both types are propositions (since we assume that $\C$ is a precategory, and therefore identity in any hom-set will be a proposition) it suffices to show that they are logically equivalent.

For
\[
\Ind^\C (h, \alpha(h)) \rightarrow \Pitype{\begin{subarray}\: w \colon O^\C \\ g \colon A^\C(w,z)  \end{subarray}} \alpha(h) \circ g = \alpha (h \circ g)
\]
note that by the definition of $\Ind$, $\Ind^\C (h, \alpha (h))$ will contain the following proposition among the cartesian product that defines it:
\[
\Pitype{\begin{subarray} \:w \colon O \\ g \colon A(w,z) \end{subarray}} \big( \alpha(h) \circ g = \alpha (h \circ g) \leftrightarrow h \circ g = h \circ g \big)
\]
Since the RHS is trivially inhabited, the above type logically entails
\[
\Pitype{\begin{subarray}\: w \colon O^\C \\ g \colon A^\C(w,z)  \end{subarray}} \alpha(h) \circ g = \alpha (h \circ g)
\]
as required.

In the other direction, the only non trivial biconditional in $\Ind^\C (h, \alpha(h))$ given the data will be
\[
\Pitype{\begin{subarray} \:w,w' \colon O^\C \\ g \colon A^\C(w,z) \\ k \colon A^\C(x,w') \end{subarray}} \big( g \circ h = k \leftrightarrow g \circ \alpha(h) = \alpha(k) \big)
\]
But this follows immediately from 
\[
\Pitype{\begin{subarray}\: w \colon O^\C \\ g \colon A^\C(w,z)  \end{subarray}} \alpha(h) \circ g = \alpha (h \circ g)
\]
as required.
\end{proof}

\begin{lemma}\label{yonlemlem2}
Let $\C$ be a precategory. For every $x,y \colon O^\C$ we have
\[
\Pitype{z \colon O^\C} A (z,x) \simeq^\C A (z,y)
\]
\end{lemma}
\begin{proof}
By Lemma \ref{indtofun} we know that $A (z,x) \simeq^\C A (z,y)$ is equivalent to the type of bijections from $A^\C(z,x)$ to $A^\C(z,y)$. Since each $\C$ is a precategory, $A^\C(z,x)$ and $A^\C(z,y)$ are sets, and therefore bijections are equivalences, which means that $A (z,x) \simeq^\C A (z,y)$ is the type of (semantic) equivalences $A^\C(z,x) \simeq A^\C(z,y)$. Exactly the same reasoning applies to $A (x,z) \simeq^\C A (y,z)$, $A (x,x) \simeq^\C A (x,y)$, $A (x,x) \simeq^\C A (y,y)$ and $A (x,y) \simeq^\C A (y,y)$ which are the rest of the conjuncts of $x \cong y$ if we follow Definition \ref{interpretationofequivalence}.
Thus, to prove the Lemma it suffices to show 
\[
\bigg[ \Pitype{z \colon O^\C} A^\C(z,x) \simeq A^\C(z,y) \bigg] \simeq \bigg[ \Pitype{z \colon O^\C} \big( A^\C (z,x) \simeq A^\C(z,y) \big) \times \big( A^\C(x,z) \simeq A^\C(y, z)  \big) \bigg]
\]
which follows from the elementary result in category theory that precomposition by an arrow induces a bijection on every hom-set iff postcomposition also does.
\end{proof}

\begin{prop}\label{unicatsat}
The type $\SMode{\Tcat}{}$ of totally saturated models of the $\Lcat$-theory of categories is equivalent to the type $\textbf{\emph{UniCat}}$ of univalent categories.
\end{prop}
\begin{proof}
By Proposition \ref{precatsat2} it suffices to show that for any precategory $\C$ and any $x,y \colon O^\C$ we have
$
\Iso^\C (x,y) \simeq x \cong y
$.
To see this observe the following series of equivalences:
\[
\xymatrix{
\Iso^\C (x,y) \ar@{=}[d]_{\text{By Definition}}  \\ 
\Sigmatype{f \colon A^\C(x,y)} \isiso(f) \ar[d]^*[@]{\simeq}_{\text{By Yoneda Lemma}} \\
\Pitype{z \colon O^\C}\: \Sigmatype{\alpha \colon A^\C(z,x) \rightarrow A^\C(z,y)} \isbijection (\alpha) \times \Pitype{\begin{subarray} \:w \colon O^\C \\h \colon A^\C(z,x) \\ g \colon A^\C (w,z) \end{subarray}} 
\alpha(h) \circ g = \alpha (h \circ g) \ar[dd]^*[@]{\simeq}_{\text{By Lemma \ref{yonlemlem}}}\\ \\
\Pitype{z \colon O^\C}\: \Sigmatype{\alpha \colon A^\C(z,x) \rightarrow A^\C(z,y)} \isbijection (\alpha) \times \Pitype{h \colon A^\C(z,x)} \Ind^\C(h, \alpha(h)) \ar[d]^*[@]{\simeq}_{\text{By Lemma \ref{indtofun}}} \\
\Pitype{z \colon O^\C} \bigg[ \big[ \Pitype{f \colon A (z,x)} \Sigmatype{g \colon A(z,y)} \Ind (f,g) \times \Pitype{h \colon A(z,y)} \Ind(f,h) \rightarrow h=g \big] \ar@{{}{ }{}}[d]^{\times} \\
\big[ \Pitype{\begin{subarray} \:f,f' \colon A^\C(z,x) \\ g,g' \colon A^\C(z,y) \end{subarray}} \Ind^\C (f,g) \times \Ind^\C (f',g') \times g = g' \rightarrow f = f') \big] \ar@{{}{ }{}}[d]^{\times} \\
\big[ \Pitype{g \colon A^\C(z,x)} \Sigmatype{f \colon A^\C(z,y)} \Ind^\C(f,g) \big] \bigg] \ar@{=}[d]\\ 
\Pitype{z \colon O^\C} A(z,x) \simeq^\C A(z,y) \ar[d]^*[@]{\simeq}_{\text{By Lemma \ref{yonlemlem2}}} \\
x \cong y
}
\]
\end{proof}


\begin{cor}[\cite{HTT}, Theorem 9.4.16]
For any univalent categories $\C$ and $\D$, the type of categorical equivalences $\C \simeq_{\text{\emph{cat}}} \D$ is equivalent to $\C =_{\emph{\UniCat}} \D$.
\end{cor}
\begin{proof}
By Theorem \ref{HSIP} what remains now is to show that in $\textbf{UniCat}$, the relation $\simeq_{\Lcat}$ is equivalent to categorical equivalence as defined in \cite{HTT}. But this is immediate by inspection.
\end{proof}

\begin{remark}
Propositions \ref{precatsat2} and \ref{unicatsat} in our opinion provide the correct taxonomy of category theory in the Univalent Foundations. In particular, we should regard precategories not as ``unsaturated'' categories but rather as the $2$-saturated $\Tcat$-models. 
This leaves two more ``category theories'' that currently have no name: the ``totally unsaturated'' models $\Mod{\Tcat}$ and the $1$-saturated models $\SMod{\Tcat}{1}$.
There are thus four different category theories in the Univalent Foundations with respect to saturation as we have defined it. They can be thought of as follows, with the obvious inclusions:
\[
\textbf{UniCat} \hookrightarrow \textbf{PreCat} \hookrightarrow \SMod{\Tcat}{1} \hookrightarrow \Mod{\Tcat}
\]
We believe this is the right point of view to take with respect to category theory (and other theories) in the Univalent Foundations.
\end{remark}

\begin{remark}
As the series of inclusions above suggests, clearly one can now ask to what extent we can have ``Rezk completion''-type operations that reverse (in appropriate ways) the above inclusions. This is clearly an avenue worth pursuing further. In particular, the following question seems to us of immediate interest: Given an $\L$-theory $\T$, for which $n \geq \heightof(\L)$ can we construct an adjoint to the inclusion 
$
\SMod{\Tcat}{n} \hookrightarrow \SMod{\Tcat}{n-1}
$? Answers to this question seem to us unlikely to be pursued at least initially for any $\L$ with $\heightof(\L) > 3$.
\end{remark}

Finally, let us mention that in principle there is no obstacle to extending the proof of Theorem \ref{HSIP} to signatures $\L$ of arbitrarily large finite height. There are however two main difficulties. Firstly, finding a sufficiently practical way of packaging the syntax of such an $\L$ so as to define the relevant notions of $\L$-homomorphism, $\L$-equivalence etc.
Secondly, proving the higher analogues of Lemma \ref{fibsurjpres2} which would require establishing equivalences between $n$-types of ever larger $n$.
Solving these difficulties seems not entirely out of reach. 
Proving a version of Theorem \ref{HSIP} for all $n > 1$ therefore seems to us an interesting open problem.

\subsection*{Acknowledgments} I thank Peter LeFanu Lumsdaine for the discussion that initiated this paper. This work was partially supported by NSF DMS-1554092 (P.I. Harry Crane).

\bibliographystyle{shortalphabetic}

\bibliography{foldssatrefs}

\end{document}